\documentclass{amsart}
\usepackage{amsmath,amsfonts,amsthm,amssymb}
\usepackage{hyperref}

\newcommand{\ind}[1]{1_{#1}}

\newcommand{\bbn}{\mathbb{N}}
\newcommand{\bbz}{\mathbb{Z}}
\newcommand{\bbc}{\mathbb{C}}
\newcommand{\bbr}{\mathbb{R}}
\newcommand{\bbq}{\mathbb{Q}}
\newcommand{\abs}[1]{\left\lvert #1\right\rvert}
\newcommand{\Abs}[1]{\lvert #1\rvert}
\newcommand{\brac}[1]{\left( #1\right)}
\newcommand{\norm}[1]{\left\lVert #1\right\rVert}
\newcommand{\td}{\,\mathrm{d}}
\newcommand{\Mod}[1]{\ (\mathrm{mod}\ #1)}
\newcommand{\A}{\mathcal{A}}
\newcommand{\B}{\mathcal{B}}
\newcommand{\C}{\mathcal{C}}
\renewcommand{\P}{\mathcal{P}}
\newtheorem{theorem}{Theorem}
\newtheorem{lemma}{Lemma}

\newtheorem{proposition}{Proposition}
\newtheorem{definition}{Definition}
\begin{document}
\title{A new upper bound for sets with no square differences}
\author{Thomas F. Bloom and James Maynard}
\date{}

\begin{abstract}
We show that if $\A\subset \{1,\ldots,N\}$ has no solutions to $a-b=n^2$ with $a,b\in \A$ and $n\geq 1$ then 
\[\lvert \A\rvert \ll \frac{N}{(\log N)^{c\log\log \log N}}\]
for some absolute constant $c>0$. This improves upon a result of Pintz-Steiger-Szemer\'edi.
\end{abstract}
%
%
%%%%%%%%%%%%%%%%%%%%%%%%%%%%%%%%%
%
%

\maketitle
%
%
%%%%%%%%%%%%%%%%%%%%%%%%%%%%%%%%%
%
%
\section{Introduction}

Answering a question of Lov\'{a}sz, S\'{a}rk\"{o}zy \cite{Sa:1978} and Furstenberg \cite{Fu:1977} independently showed that any set of integers whose difference set contains no non-zero squares must have asymptotic density zero. S\'{a}rk\"{o}zy's proof was based on the circle method, and actually gave the quantitative bound that if $\A\subseteq\{1,\dots,N\}$ has no non-zero square differences then $|\A|\le N/(\log{N})^{1/3+o(1)}$, whereas Furstenberg's result relied on ergodic theory. There have since been a variety of proofs of the qualitative result $|\A|=o(N)$; we refer the reader to the introduction of \cite{Ri} for more details.

S\'{a}rk\"{o}zy's argument was refined by  Pintz, Steiger, and Szemer\'{e}di \cite{PiStSz:1988} who improved the upper bound on the size of $\A\subseteq\{1,\dots,N\}$ with no non-zero square differences to
\begin{equation}
|\A|\ll \frac{N}{(\log N)^{c\log \log \log \log N}}
\label{eq:PSS}
\end{equation}
for some absolute constant $c>0$. (Here we use Vinogradov's notation $X\ll Y$ to mean that $X\leq CY$ for some absolute constant $C>0$.) One interesting feature of \eqref{eq:PSS} is that it is a noticeably stronger bound than what is currently known for Roth's Theorem on three term arithmetic progressions \cite{BloomSisask}, despite both proofs following a Fourier-analytic density increment argument.

In this paper we improve the upper bound \eqref{eq:PSS} for the size of sets of integers with no square differences.
%
%
%%%%%%%%%%%%%%%%%%%%%%%%%%%%%%%%%
%
%
\begin{theorem}\label{thmain}
Let $N$ be sufficiently large. If $\A\subset \{1,\ldots,N\}$ has no solutions to $a-b=n^2$ with $a,b\in \A$ and $n\geq 1$, then 
\[\abs{\A}\ll \frac{N}{(\log N)^{c\log \log\log N}}\]
for some absolute constant $c>0$. 
\end{theorem}
%
%
%%%%%%%%%%%%%%%%%%%%%%%%%%%%%%%%%
%
%
Our proof of Theorem \ref{thmain} follows a Fourier-analytic density increment argument as with previous approaches, but is actually more direct (and, we hope, simpler) than the approach of Pintz-Steiger-Szemer\'{e}di. The key new tool in our work is an upper bound for the additive energy of sets of rationals with small denominator, which may be of independent interest. To state this we recall that the $2m$-fold additive energy of a set $\B$ is given by
\begin{align*}
E_{2m}(\B):= \lvert \{ (b_1,\ldots,b_{2m})\in \B^{2m} : b_1+\cdots+b_m=b_{m+1}+\cdots+b_{2m}\}\rvert.
\end{align*}
We also introduce the notation
\begin{align*}
\bbq_{=q} &:= \left\{ \frac{a}{q}\in [0,1] : 1\leq a\leq q\textrm{ and }\gcd(a,q)=1\right\},\\
\bbq_{\leq Q} &:= \bigcup_{1\leq q\leq Q}\bbq_{=q},
\end{align*}
to denote the set of reduced rationals in $[0,1]$ with denominator precisely $q$, and for the set of all rationals with denominator at most $Q$. Our additive energy result is then the following.
%
%
%%%%%%%%%%%%%%%%%%%%%%%%%%%%%%%%%
%
%
\begin{theorem}\label{th:main}
Let $Q\geq 4$ and $m\geq 2$. Suppose that $\B\subset \bbq_{\leq Q}$ is such that there is $n\geq 1$ with $\abs{\B\cap \bbq_{=q}}\leq n$  for any $1\leq q\leq Q$. Then we have
\[
E_{2m}(\B) \leq (\log Q)^{C^m} (Qn)^m
\]
for some absolute constant $C>0$.
\end{theorem}
%
%
%%%%%%%%%%%%%%%%%%%%%%%%%%%%%%%%%
%
%
We note that there is a trivial lower bound $E_{2m}(\B)\ge |\B|^m$ from diagonal solutions where $b_{i}=b_{i+m}$ for $1\leq i\leq m$. If $\B$ contains $n$ rationals with denominator $q$ for each $q\in [Q/2,Q]$ then $|\B|\gg nQ$ and we see that Theorem \ref{th:main} gives an upper bound of the form $|\B|^m (\log{|\B|})^{C^m}$, and so the contribution from the diagonal terms is comparable to the whole contribution. Thus sets of rationals with small distinct denominators have similar additive energy estimates to dissociated sets where the only solutions to $b_1+\dots+b_m=b_{m+1}+\dots+b_{2m}$ are the diagonal ones. 

Dissociated sets have been used in additive combinatorics since at least the work of Chang \cite{Chang}, and Theorem \ref{th:main} allows one to extend Chang's ideas to sets whose large Fourier coefficients are close to rationals with small denominators, as is the situation in the Furstenberg-S\'{a}rk\"{o}zy  problem. The original argument of Pintz-Steiger-Szemer\'{e}di can be viewed as showing that there is a lack of additive structure in the rationals which make up the large Fourier coefficients, but Theorem \ref{th:main} allows for a more efficient and direct use of this idea. Indeed, the original argument of \cite{PiStSz:1988} proves (implicitly) a lower bound for the size of the $m$-fold sumset, namely something of the strength of
$\abs{m\B}\gg_{m,\epsilon} \abs{\B}^{m-\epsilon}$, which follows from Theorem \ref{th:main} and $|m\B|\gg |\B|^{2m}/E_{2m}(\B)$. An important feature of Theorem \ref{th:main} for our work is that it remains non-trivial even with $m$ as large as a small multiple of $\log\log{Q}$.

Clearly one can have sets $\B\subset\bbq_{\leq Q}$ with large additive energy if many elements of $\B$ have the same denominator, and so the restriction $|\B\cap\bbq_{=q}|\leq n$ is natural for this problem.
\subsubsection*{Other results and generalizations}

For comparison, the best known lower bound for the size $r(N)$ of the largest set $\A\subset\{1,\dots,N\}$ with no non-zero square differences is much smaller than the upper bound in Theorem~\ref{thmain}. Ruzsa \cite{Ru:1984} gives a construction that shows in particular that $r(N) \gg N^{0.73}$. The constant in the exponent here has been slightly improved by Lewko \cite{Le:2015}, but even whether $r(N)\gg N^{3/4}$ is open. We do not know where the truth lies, and whether the true order of magnitude of $r(N)$ is $N^{1-o(1)}$ or $N^{1-c}$ for some absolute constant $c>0$ remains a fascinating open problem.

For the analogous problem in the function field case, where $\bbz$ is replaced by the polynomial ring $\mathbb{F}_q[t]$ over some finite field $\mathbb{F}_q$, much stronger quantitative bounds are known. Using the polynomial method, Green \cite{Gr:2017} has recently shown that if $\A\subset \mathbb{F}_q[t]_{\mathrm{deg}<n}$ contains no non-zero square differences then
\begin{equation}\label{fqtgreen}
\lvert \A \rvert \ll q^{(1-c(q))n},
\end{equation}
where $c(q)>0$ is some constant depending only on $q$. Since $\mathbb{F}_q[t]_{\mathrm{deg}<n}$ has size $q^n$, this bound is analogous to a bound of the shape $r(N) \ll N^{1-c}$ in the integer case. The polynomial method used by Green is very different to the analytic arguments used in this paper, and depends in a fundamental way on the bounded characteristic of $\mathbb{F}_q[t]$. 

The method of Pintz, Steiger, and Szemer\'{e}di \cite{PiStSz:1988} has been generalised to yield a similar bound for related problems. This was done for sets without differences of the form $n^k$ for any fixed $k\geq 3$ by Balog, Pelikan, Pintz, and Szemer\'{e}di \cite{BaPePiSz:1994}, and then recently by Rice \cite{Ri} to differences of the form $f(n)$ where $f\in\bbz[x]$ is any intersective polynomial\footnote{A polynomial $f\in\bbz[x]$ is intersective if it is non-zero and for every $q\in\mathbb{N}$ there is $n\in\bbz$ such that $q\mid f(n)$.} of degree at least 2. These proofs directly extend the method of \cite{PiStSz:1988}, and as such it seems likely that one could combine the ideas of \cite{BaPePiSz:1994} and \cite{Ri} with those in this paper to obtain a quantitative bound of strength comparable to Theorem~\ref{thmain} for these generalisations; we do not address these questions here.

Recent work of the second author \cite{Maynard} showed that any system of polynomials simultaneously attain values with small fractional parts. There are various similarities with this work (a density increment argument enhanced by there being few solutions to linear equations involving rationals with small denominator), but there the problem was more structured which allowed for an almost optimal bound of the form $x^{1-c}$, whereas in this situation we are forced to consider much more arbitrary sets $\A$.

An upper bound for the additive energy of sets of well-distributed rationals similar (qualitatively) to Theorem~\ref{th:main} has also been applied within theoretical computer science, where it was used by Bourgain, Dilworth, Ford, Konyagin, and Kutzarova \cite{BDFKK} to construct matrices satisfying the Restricted Isometry Property. It follows from their Lemma 5, for example, that if $\B\subset \bbq_{\leq Q}$ is such that for any $1\leq q\leq Q$ we have $\abs{\B\cap \bbq_{=q}}\leq 1$, then, for any $\epsilon>0$, we have $E_{2m}(\B)\ll_{m,\epsilon}Q^\epsilon\abs{\B}^m$, where the dependence on $m$ and $\epsilon$ is unspecified. It is vital for our purposes that we explicitly control the dependence on $m$ and $\epsilon$.
%
%
%%%%%%%%%%%%%%%%%%%%%%%%%%%%%%%%%
%
%
\subsubsection*{Acknowledgements}

T.B. would like to thank Julia Wolf for many helpful discussions regarding the original proof of Pintz, Steiger, and Szemer\'{e}di (and recommends the first chapter of \cite{Wo} for a very readable exposition), and would also like to thank Alex Rice for pointing out an error in an earlier draft of this paper. J.M. would like to thank Ben Green for introducing him to this question. T.B. was supported by a
postdoctoral grant funded by the Royal Society. J.M. was supported by a Royal Society Wolfson Merit Award, and funding from the European Research Council (ERC) under the European Union's     Horizon 2020 research and innovation programme (grant agreement No 851318).
%
%
%%%%%%%%%%%%%%%%%%%%%%%%%%%%%%%%%
%
%
\section{Outline}

In this section we sketch how Theorem \ref{th:main} can be used to give Theorem \ref{thmain}. As mentioned in the introduction, our proof is similar to the original work of S\'{a}rk\"{o}zy \cite{Sa:1978} (and its later refinements) in that we base our argument on a density increment argument coming out of the circle method. Our Theorem \ref{th:main} allows us to show that no set $\A$ can have many large Fourier coefficients which are rationals with small distinct denominators, and this is the key which allows us to have a more efficient density increment argument than that of \cite{PiStSz:1988}. 

First we recall the basic setup. If $\A\subset\{1,\dots,N\}$ has no non-zero square differences and density $\alpha=|\A|/N$, then by the circle method
\begin{align*}
0&=|\{(a_1,a_2,n):\,a_1-a_2=n^2,\,a_1,a_2\in\A,\,1\le n\le N^{1/2}\}|\\
&=\int_0^1 |\widehat{1}_{\A}(\gamma)|^2\widehat{1}_{\square}(\gamma)\mathrm{d}\gamma,
\end{align*}
where $\widehat{1}_{\A}(\gamma)=\sum_{a\in \A}e(a\gamma)$ is the Fourier transform of the set $\A$, and similarly $\widehat{1}_{\square}$ is the Fourier transform of squares in $[1,N]$. Comparing this to $\alpha^2 N^{3/2}$, which is the expected count of solutions in a random set of density $\alpha$, we find 
\[
\int_0^1 |\widehat{g}_{\A}(\gamma)|^2|\widehat{1}_{\square}(\gamma)|\mathrm{d}\gamma\gg \alpha^2 N^{3/2},
\]
where $g_{\A}=1_{\A}-\alpha 1_{[N]}$ is the balanced function of $\A$. If $\gamma\approx a/q$ then $\lvert\widehat{1}_{\square}(\gamma)\rvert\approx N^{1/2}/q^{1/2}$, and so by the pigeonhole principle there must be some $Q$ for which
\[
\int_{\substack{\gamma\approx a/q\\Q\le q\le 2Q}}|\widehat{g}_{\A}(\gamma)|^2\mathrm{d}\gamma\gg \frac{\alpha^2 N Q^{1/2}}{(\log{Q})^2},
\]
In particular there must be some parameter $B$ such that there are `many' ($\gg Q^{1/2}B^2/(\log{Q})^3$) rationals of $a/q$ with $q\in[Q,2Q]$ for which $\widehat{g}_{\A}(a/q)$ is large ($\gg |\A|/B$). Parseval's bound implies that this can only happen if $Q\le \alpha^{-O(1)}$, and so these rationals have small denominators. The basic density increment strategy is then to deduce that this implies that there is a large arithmetic progression on which $\A$ has density $\alpha+\rho\alpha$ (for some suitable $\rho>0$), so this argument can then be iterated.

If there were at least $\rho^2 B^2$ different rationals $\gamma$ with the same denominator $q$ where $|\widehat{g}_\A(\gamma)|\gg |\A|/B$, then the existence of such an arithmetic progression follows from
\[
\sum_{b\Mod{q}}\Bigl(|\{a\in\A:\,a\equiv b\Mod{q}\}|-\frac{|\A|}{q}\Bigr)^2=\frac{1}{q}\sum_{c\Mod{q}}\Bigl|\widehat{g}_{\A}\Bigl(\frac{c}{q}\Bigr)\Bigr|^2\gg \frac{\rho^2 |\A|^2}{q}.
\]
Thus we may assume that for any $q\in[Q,2Q]$ there are $O(\rho^2 B^2)$ rationals $\gamma$ with denominator $q$ where $|\widehat{g}_{\A}(\gamma)|\gg |\A|/B$. So far our approach has followed previous works, but now we note that if $\B$ is this set of rationals with $|\widehat{g}_{\A}(b)|\gg|\A|/B$, then a variation of the proof of Chang's lemma shows roughly that for any choice of $m\ge1$
\begin{align*}
\frac{|\A|\cdot |\B|}{B}\ll \sum_{b\in \B}|\widehat{g}_{\A}(b)|
&\le |\A|^{1-1/2m}N^{1/2m}E_{2m}(\B)^{1/2m}.
\end{align*}
In particular, there cannot be a large set $\B$ with very small additive energy whilst also having the Fourier transform $\widehat{g}_\A$ of $\A$ large on all elements of $\B$. We can now apply Theorem \ref{th:main} to bound $E_{2m}(\B)$, since $\B$ is a set of rationals of small denominator with $O(\rho^2 B^2)$ elements having any given denominator. This rearranges to give (for any choice of $m$)
\[
|\B|\ll \rho\alpha^{-1/2m}(\log Q)^{C^m}Q^{1/2} B^2.
\]
This contradicts the lower bound $|\B|\gg Q^{1/2}B^2/(\log{Q})^3$ if
\[
\rho\approx  \frac{\alpha^{1/2m}}{(\log Q)^{C^m+3}}.
\]
Choosing $m=c\log\log({1/\alpha})$ for some suitably small constant $c>0$ and recalling $Q\ll \alpha^{-O(1)}$, this implies that there is an arithmetic progression $\P\subseteq \{1,\dots,N\}$ with $|\P|\ge \alpha^{O(1)}N$ on which $\A$ has density $\alpha(1+\rho)$, where
\[
\rho\approx  \exp\Bigl(-C\frac{\log{1/\alpha}}{\log\log{1/\alpha}}\Bigr)
\]
for some absolute constant $C>0$. Iterating this statement then gives Theorem \ref{thmain}.

Theorem \ref{th:main} is established using a quite separate argument. Although it can be deduced from a direct combinatorial approach based on splitting according to suitable greatest common divisors we use an iterative argument which we hope is cleaner.
%
%
%%%%%%%%%%%%%%%%%%%%%%%%%%%%%%%%%
%
%
\section{Notation}

We begin by establishing the basic notation that we will use.  For any $N\geq 1$ we use $[N]$ to denote the set $\{1,\ldots,N\}$. We fix throughout our proof some large integer $N$ (large enough in particular such that $\log\log\log N\geq 4$, say). For functions $f:\bbz\to\bbc$ we define the Fourier transform $\widehat{f}:[0,1]\to \bbc$ by
\[\widehat{f}(\gamma) = \sum_{n\in \bbz}f(n)e(\gamma n),\]
where $e(x)=e^{2\pi i x}$. We define the convolution of two functions $f,g:\bbz\to\bbc$ by 
\[(f\ast g)(n)= \sum_{m\in\bbz}f(m)g(n-m)\]
and use $f^{(\ast m)}(x)=(f\ast \dots \ast f)(x)$ to denote the $m$-fold iterated convolution of $f$.  Without subscript, the notation $\norm{\gamma}$ denote the distance of $\gamma\in\bbr$ from the nearest integer, while $\norm{f}_2=(\sum_{x\in\bbz}|f(x)|^2)^{1/2}$ and $\norm{f}_{\infty}=\sup_{x\in\bbz}|f(x)|$ denotes the usual $L^2$ and $L^\infty$ norms. 

We write $\tau_3(n)$ to denote the ternary divisor function $\sum_{abc=n}1$.

% To avoid confusion, $(d,f)$ will always denote the ordered pair, and $\mathrm{gcd}(d,f)$ will be used for the greatest common divisor of $d$ and $f$.
%
%
%%%%%%%%%%%%%%%%%%%%%%%%%%%%%%%%%
%
%
\section{Addition of rational numbers and the proof of Theorem \ref{th:main}}

In this section we prove Theorem \ref{th:main}. This section is essentially self-contained and can be read independently of the rest of the paper.

Theorem \ref{th:main} follows quickly from the more technical Proposition \ref{th:rat}. To state this we require some more notation. For any function $\omega:\bbn\to \bbr$ we define the maximal average function of $\omega$ by 
\begin{equation}
M(\omega;X):= \max_{1\leq x\leq X}\frac{1}{x}\sum_{n\leq x}\omega(n),\label{eq:MDef}
\end{equation}
and the logarithmic maximal average by
\begin{equation}
M_{\mathrm{log}}(\omega;X):=\max_{2\leq x\leq X}\frac{1}{\log x}\sum_{n\leq x}\frac{\omega(n)}{n}.\label{eq:MLDef}
\end{equation}
We recall that $\tau_3(n)$ is the ternary divisor function $\sum_{abc=n}1$. Our technical bound on the additive energy is the following.

\begin{proposition}[Rationals with small denominators have small additive energy]\label{th:rat}
Let $m\geq 2$. Suppose that $\B\subset \bbq_{\leq Q}$ and $n\geq 1$ is such that for any $1\leq q\leq Q$ we have $\abs{\B\cap \bbq_{=q}}\leq n$. Then we have the upper bound
\[E_{2m}(\B) \leq (m\log Q M_{\mathrm{log}}(\tau_3^{2m};Q))^{O(m)}M(\tau_3^{2m-2};Q)(Qn)^m.\]
\end{proposition}
%
%
%%%%%%%%%%%%%%%%%%%%%%%%%%%%%%%%%
%
%
\begin{proof}[Proof of Theorem \ref{th:main} assuming Proposition \ref{th:rat}]
By Rankin's trick, we have the standard divisor function bound
\[
M(\tau_3^k;X)\ll (\log{X})M_{\mathrm{log}}(\tau_3^k;X) \ll \prod_{p\le x} \Bigl(1-\frac{1}{p}\Bigr)^{-3^k}\ll (\log X)^{3^k}.
\]
Therefore Proposition \ref{th:rat} gives
\[
E_{2m}(\B)\ll m^{O(m)}(\log Q)^{O(m9^m)} (Q n)^m.
\]
Simplifying the exponents gives Theorem \ref{th:main}.
\end{proof}
Proposition~\ref{th:rat} will be proved via an iterative application of the following lemma. Roughly speaking, it says that if $B\subset \bbq_{\leq L}$ is spread evenly between different denominators, then for any sets $\A,\C$ we have $\langle \ind{\A},\, \ind{\B}\ast\ind{\C}\rangle \ll (\abs{\A}\abs{\B}\abs{\C})^{1/2}$. This should be compared to the trivial bound of $(\abs{\A}\abs{\C})^{1/2}\abs{\B}$. To attain the quantitative strength of Theorem~\ref{th:rat} we will take care to prove an explicit weighted form of this inequality. To state this lemma precisely we make the following definition.
%For any arithmetic function $\omega: \bbn\to \bbc$ we define $\tilde{\omega}:\bbq\to \bbc$ by $\tilde{\omega}(a/q):=\omega(q)$ where $a\neq 0$ and $q\geq 1$ satisfies $\gcd(a,q)=1$, and letting $\tilde{\omega}(0):=1$. 
%
%
%%%%%%%%%%%%%%%%%%%%%%%%%%%%%%%%%
%
%
\begin{definition}
An arithmetic function $\omega:\bbn\to \bbr_{\ge 0}$ is sub-multiplicative if $\omega(ab)\leq \omega(a)\omega(b)$ for all $a,b\in \bbn$ and whenever $d\mid n$ we have $\omega(d)\leq \omega(n)$. 
\end{definition}
%
%
%%%%%%%%%%%%%%%%%%%%%%%%%%%%%%%%%
%
%
\begin{lemma}[(Few solutions to linear equations in rationals with small denominators)]\label{le:rat}
Let $Z\geq 1$ be some integer. Let $\A\subset \bbq\cap (0,Z]$ and $\C\subset \bbq$. Suppose that $\B\subset \bbq_{\leq Q}$ is such that, for any $1\leq \ell \leq Q$, we have $\abs{\B\cap \bbq_{=\ell}}\leq n$. Let $\omega:\bbn\to\bbr_{\geq 0}$ be some sub-multiplicative function. Then
\[\sideset{}{'}\sum_{\substack{a/k-b/\ell=c/q\\ a/k\in\A, \,b/\ell\in\B,\,c/q\in \C}}\omega(k)\ll 
\brac{QnZ(\log{Q}) M_{\mathrm{log}}(\omega\tau_3;Q)\sideset{}{'}\sum_{c/q\in \C}\omega(q)\tau_3(q)^2\sideset{}{'}\sum_{a/k\in \A}\omega(k)}^{1/2}.\]
Here we use $\sum'$ to indicate that the fractions $a/k,b/\ell,c/q$ are all reduced (i.e. $\gcd(a,k)=\gcd(b,\ell)=\gcd(c,q)=1$).
\end{lemma}
We note that the summation on the left hand side in Lemma \ref{le:rat} can also be written as $\sum_{x\in\C}(\widetilde{\omega}\ind{\A}\ast \ind{-\B})(x)$, where $\widetilde{\omega}(a/q)=\omega(q)$ for $\gcd(a,q)=1$. If $\omega\approx 1$ and $|\B|\approx Qn$ then since $\tau_3$ is typically quite small the bound on the right hand side is roughly of size $(\abs{\A}\abs{\B}\abs{\C})^{1/2}$.

%
%
%%%%%%%%%%%%%%%%%%%%%%%%%%%%%%%%%
%
\begin{proof}
Throughout the proof we use $\sum'$ to indicate that the fractions in the summation are reduced.

We claim that for any choice of parameter $T>0$, there is a decomposition of $\A\times \B$ into two sets $\mathcal{E}_1$ and $\mathcal{E}_2$ such that, if we let
\[F_i(x):= \sideset{}{'}\sum_{(a/k,b/\ell)\in \mathcal{E}_i} \omega(k)1_{a/k-b/\ell=x},\]
then we have
\begin{equation}\label{graphbound}
\sideset{}{'}\sum_{c/q\in \C}F_1\Bigl(\frac{c}{q}\Bigr) \leq \frac{QnZ\log Q}{T}M_{\mathrm{log}}(\omega\tau_3;Q)\sideset{}{'}\sum_{c/q\in \C}\omega(q)\tau_3(q)^2
\end{equation}
and
\begin{equation}\label{graphbound2}
\quad\sideset{}{'}\sum_{c/q\in \C}F_2\Bigl(\frac{c}{q}\Bigr)  \leq T\sideset{}{'}\sum_{a/k\in \A}\omega(k).
\end{equation}
The lemma now follows from this claim by choosing 
\[T=\brac{\frac
{Qn Z(\log Q) M_{\mathrm{log}}(\omega\tau_3;Q) \sideset{}{'}\sum_{c/q\in C}\omega(q)\tau_3(q)^2}
{\sideset{}{'}\sum_{a/k\in \A}\omega(k)}}^{1/2},\]
since
\[\sideset{}{'}\sum_{\substack{a/k-b/\ell=c/q\\ a/k\in\A, \,b/\ell\in\B,\,c/q\in \C}}\omega(k)= \sideset{}{'}\sum_{c/q\in C}\brac{F_1(c/q)+F_2(c/q)}.\]
Thus we are left to establish the claim by constructing the sets $\mathcal{E}_1$ and $\mathcal{E}_2$. We consider the complete bipartite graph $G$ with vertex sets $\A$ and $\B$. We first colour the edges of $G$ by giving the edge $(a/k,b/\ell)\in\A\times \B$ the colour $C(a/k,b/\ell)=(d,f)\in\bbz^2$, where
\[d = \mathrm{gcd}(k,\ell)\quad\textrm{ and }\quad f = \mathrm{gcd}\brac{\frac{a\ell-bk}{d},d}.\]
We then say that the colour $(d,f)$ is `popular at $a/k$' if 
 \[|\{ b/\ell\in \B : C(a/k,b/\ell)=(d,f)\}|\geq \frac{T}{\tau_3(k)}.\]
 We say that an edge $(a/k,b/\ell)\in\A\times \B$ is `popular' if its colour is popular at $a/k$, then let $\mathcal{E}_1\subset \A\times \B$ be the set of all popular edges, and let $\mathcal{E}_2$ be the remaining edges $(\A\times \B)\backslash \mathcal{E}_1$.

The bound in \eqref{graphbound2} now follows easily. Indeed, if $a/k\in\A$ has an edge coloured $(d,f)$ connected to $a/k$ then $f|d|k$, so there are at most $\tau_3(k)$ possible different colours of edges connected to $a/k$. Since $\mathcal{E}_2$ only consists of the edges which are not popular, for any colour $(d,f)$ there are at most $T/\tau_3(k)$ vertices $b/\ell$ connected to $a/k$ by an edge of colour $(d,k)$. Thus there are at most $T$ vertices $b/\ell$  in $\B$ connected to any given vertex in $\A$ by an edge in $\mathcal{E}_2$, and so
\begin{align*}
\sideset{}{'}\sum_{ c/q\in \C} F_2(c/q)\leq \sideset{}{'}\sum_{a/k\in \A}\omega(k)\,\sideset{}{'}\sum_{b/\ell\in \B}1_{(a/k,b/\ell)\in \mathcal{E}_2}\leq T\sideset{}{'}\sum_{a/k\in A}\omega(k).
\end{align*}
This gives \eqref{graphbound2}.

%
%
%To obtain \eqref{graphbound} we will bound $F_1(c/q)$ for each choice of $c$ and $q$. Suppose that
%\[\frac{c}{q} = \frac{a}{k}-\frac{b}{\ell} = \frac{a\ell - bk}{k\ell}.\]
%Let $d=\mathrm{gcd}(k,\ell)$ and, for convenience, write $k'=k/d$ and $\ell'=\ell/d$ (so that, in particular, $\gcd(k',\ell')=1$). We therefore have
%\[\frac{c}{q} = \frac{a\ell'-bk'}{k'\ell'd}.\]
%Since $k'\ell'$ is prime to the numerator, we must have $q=k'\ell' e$ for some $e\geq 1$. Fix some such decomposition, so that in particular now $k'$ and $\ell'$ are fixed. Since $c$ is fixed, this implies in particular that
%\[\frac{a\ell'-bk'}{d}=\frac{c}{e}\]
%is fixed. If we fix a possible $f=\mathrm{gcd}(a\ell'-bk',d)$, where $1\leq f\leq L$, then this also fixes $d=ef$, and hence also $k$ and $l$. Let $R_{d,f,k}$ count the number of distinct possibilities for $a\pmod{f}$ such that the colour $(d,f)$ is popular at some $a/k\in \A$. We will first show that, for any pair $(d,f)$ and $k$, we have 
%
It remains to establish \eqref{graphbound}. Given a choice of $d,f$ and $k$, let $R_{d,f,k}$ count the number of distinct possibilities for $a\pmod{f}$ such that the colour $(d,f)$ is popular at some $a/k\in \A$. We will first show that, for any pair $(d,f)$ and $k$, we have 
\begin{equation}\label{est2}
R_{d,f,k}\leq \frac{Qn}{dT}\tau_3(k).
\end{equation}
Let $\A_{d,f,k}\subset \A$ be some subset representing the $R_{d,f,k}$ different possibilities. Therefore 
\begin{enumerate}
\item The colour $(d,f)$ is popular at each $a/k\in \A_{d,f,k}$.
\item If $a/k,a'/k\in A_{d,f,k}$ then $a\not\equiv a'\pmod{f}$.
\item For each $a'/k\in \A$ such that $(d,f)$ is popular at $a'/k$, there is $a/k\in \A_{d,f,k}$ such that $a\equiv a'\pmod{f}$.
\item $R_{d,f,k}=\abs{\A_{d,f,k}}$. 
\end{enumerate}
By the definition of edges being popular at $a/k$, we have
\[R_{d,f,k}\frac{T}{\tau_3(k)}\leq \sideset{}{'}\sum_{a/k\in \A_{d,f,k}}\sideset{}{'}\sum_{\substack{b/\ell \in \B\\ C(\frac{a}{k},\frac{b}{\ell})=(d,f)}}1.\]
The key observation is that each $b/\ell\in \B$ appears at most once in total on the right-hand side, since if $C(a_1/k,b/\ell)=C(a_2/k,b/\ell)=(d,f)$, then we must have 
\[\gcd\Bigl(\frac{a_1\ell-b k}{d},d\Bigr)=\mathrm{gcd}\Bigl(\frac{a_2\ell-bk}{d},d\Bigr)=f,\qquad \gcd(k,\ell)=d.\]
In particular $a_1\ell/d\equiv a_2\ell/d\pmod{f}$. Since $\gcd(\ell,k)=d$, we see that $\gcd(\ell/d,f)=1$, and so $a_1\equiv a_2\pmod{f}$. It follows that 
\[
R_{d,f,k}\frac{T}{\tau_3(k)}\leq  \sum_{\substack{\ell\leq Q\\ d\mid l}}\abs{\B\cap \bbq_{=\ell}}\leq \frac{n Q}{d},\]
and the estimate \eqref{est2} follows immediately. 

We now establish \eqref{graphbound} by bounding $F_1(c/q)$ for each $c/q$ separately. Given a choice of $c/q$ (with $\gcd(c,q)=1$) we see that if $a,k,b,\ell$ are such that $\gcd(a,k)=\gcd(b,\ell)=1$ and
\[
\frac{c}{q}=\frac{a}{k}-\frac{b}{\ell},
\]
then $q=\ell' k' e$ and $c=(a\ell'-bk')/f$, where
\[
k'=\frac{k}{\gcd(k,\ell)},\quad \ell'=\frac{\ell}{\gcd(k,\ell)},\quad \gcd(k,\ell)=e f, \quad f=\gcd(a \ell'-b k',\gcd(k,\ell)).
\]
Thus, given a choice of $c,k',\ell',e,f$ with $k'\ell'e=q$ and $1\leq f\leq Q$, there is a unique choice of $k=k'ef$ and $\ell=\ell'ef$. Moreover, $a\ell'\equiv c f\Mod{k'}$, so $a$ is fixed modulo $k'$. There are at most $e$ choices of $a\Mod{e}$ and at most $R_{ef,f,k'ef}$ choices of $a\Mod{f}$, so at most $e R_{ef,f,k'ef}$ choices of $a\Mod{k}$. If we then fix the integer part of $a/k$, of which there are at most $Z$ choices, then we have determined $a$. Given such a choice of $a$, $b$ is uniquely determined since $b=\ell(c/q-a/k)$. It follows that 
\begin{equation}\label{bound1}
F_1(c/q)\leq Z\sum_{k'\ell'e=q}\sum_{1\leq f\leq L}\omega(k'ef)e R_{ef,f,k'ef}.
\end{equation}
Using \eqref{est2} and the sub-multiplicativity of $\omega$ and $\tau_3$, this is bounded above by 
\begin{align*}
\frac{QnZ}{T}\sum_{k'\ell'e=q}\omega(k'e)\tau_3(k'e)\sum_{1\leq f\leq Q}\frac{\omega(f)\tau_3(f)}{f}
&\le \frac{QnZ\log{Q}}{T}M_{\mathrm{log}}(\omega\tau_3;Q)\omega(q)\tau_3(q)^2.
\end{align*}
Summing this over $c/q\in\C$ then gives \eqref{graphbound}, and so completes the proof.
\end{proof}
%
%
%%%%%%%%%%%%%%%%%%%%%%%%%%%%%%%%%
%
%
We will actually use a weighted version of Lemma \ref{le:rat}, which follows immediately by a dyadic decomposition of the support of the weights.
%
%
%%%%%%%%%%%%%%%%%%%%%%%%%%%%%%%%%
%
%
\begin{lemma}\label{lem-toref}
Let $Z\geq 1$ be some integer. Let $f:\bbq\to \bbz_{\ge 0}$ and $g:\bbq\cap(0,Z]\to \bbz_{\ge 0}$ be functions with finite support such that $\norm{f}_\infty,\norm{g}_\infty\leq X$. Suppose that $\B\subset \bbq_{\leq Q}$ is such that, for any $1\leq \ell \leq Q$, we have $\abs{\B\cap \bbq_{=\ell}}\leq n$. Then, for any sub-multiplicative function $\omega:\bbn\to \bbr_{\geq 0}$,
\begin{align*}
\sideset{}{'}\sum_{\substack{a/k-b/\ell=c/q\\ b/\ell\in\B}}\omega(k)g\Bigl(\frac{a}{k}\Bigr)f\Bigl(\frac{c}{q}\Bigr)&\ll 
C \Biggl(\sideset{}{'}\sum_{c/q\in \C}\omega(q)\tau_3(q)^2f\Bigl(\frac{c}{q}\Bigr)^2\Biggr)^{1/2}\Biggl(\sideset{}{'}\sum_{a/k\in \A}\omega(k)g\Bigl(\frac{a}{k}\Bigr)^2\Biggr)^{1/2},
\end{align*}
where
\[
C:=(\log{X})\Bigl(QnZ(\log{Q}) M_{\mathrm{log}}(\omega\tau_3;Q)\Bigr)^{1/2}.
\]
Here we use $\sum'$ to indicate that the fractions $a/k,b/\ell,c/q$ are all reduced (i.e. $\gcd(a,k)=\gcd(b,\ell)=\gcd(c,q)=1$).
\end{lemma}
%
%
%%%%%%%%%%%%%%%%%%%%%%%%%%%%%%%%%
%
%
\begin{proof}[Proof of Lemma~\ref{lem-toref}]
We decompose the support of $f$ into $\C_j$ for $j\geq 0$, where
\[\C_j = \{ x: 2^j\leq f(x) < 2^{j+1}\},\]
and similarly decompose the support of $g$ into $\A_i$. Using this decomposition we have
\[\sideset{}{'}\sum_{\substack{a/k-b/\ell=c/q\\ b/\ell\in\B}}\omega(k)g\Bigl(\frac{a}{k}\Bigr)f\Bigl(\frac{c}{q}\Bigr)
\ll\sum_{0\leq i,j \leq \log{X}}2^{i+j}\sideset{}{'}\sum_{\substack{a/k-b/\ell=c/q\\ a/k\in\A_i,\,b/\ell\in\B,\,c/q\in\C_j}}\omega(k).\]
Applying Lemma~\ref{le:rat} to each summand gives the upper bound
\[\ll  \sum_{0\leq i,j \leq \log{X} }2^{i+j}\brac{Q n Z(\log{Q})M_{\mathrm{log}}(\omega\tau_3;Q)\sideset{}{'}\sum_{c/q\in \C_j}\omega(q)\tau_3(q)^2\sideset{}{'}\sum_{a/k\in \A_i}\omega(k)}^{1/2}.\]
Lemma~\ref{lem-toref} now follows by the Cauchy-Schwarz inequality. 
\end{proof}
%
%
%%%%%%%%%%%%%%%%%%%%%%%%%%%%%%%%%
%
%
The proof of Proposition~\ref{th:rat} may now proceed via induction. 
%
%
%%%%%%%%%%%%%%%%%%%%%%%%%%%%%%%%%
%
%
\begin{proof}[Proof of Proposition~\ref{th:rat}]
Again, we use $\sum'$ to indicate that the summation is restricted to reduced fractions. We will show that, for any $t\geq 0$ and $1\le j\le m$,  we have
\begin{equation}
\sideset{}{'}\sum_{c/q\in\mathbb{Q}} \tau_3(q)^{2t}\ind{\B}^{(\ast j)}(c/q)^2 \ll (m\log Q)^3(Q n M_{2t+1}) \sideset{}{'}\sum_{c/q\in\mathbb{Q}} \tau_3(q)^{2t+2}\ind{\B}^{(\ast j-1)}(c/q)^2,
\label{eq:InductionTarget}
\end{equation}
where we recall $\ind{\B}^{(\ast m)}(x)=\sum_{x_1+\dots +x_m=x}\ind{\B}(x_1)\cdots \ind{\B}(x_m)$ is the $m$-fold convolution of $\ind{\B}$, and where
\[M_t := M_{\mathrm{log}}(\tau_3^{t};Q).\]
Repeatedly applying \eqref{eq:InductionTarget} $m-1$ times gives
\begin{align*}
E_{2m}(\B)&=\sideset{}{'}\sum_{c/q\in\mathbb{Q}} \ind{\B}^{(\ast m)}(c/q)^2\\
&\ll (m\log Q)^{3(m-1)}(Q n)^{m-1}(M_1\cdots M_{2m-3})\sideset{}{'}\sum_{c/q\in\mathbb{Q}} \tau_3(q)^{2m-2}\ind{\B}(c/q)^2.
\end{align*}
Since $|\B\cap\mathbb{Q}_{=q}|\le n$, we have
\[\sideset{}{'}\sum_{c/q\in\mathbb{Q}} \tau_3(q)^{2m-2}\ind{\B}(c/q)^2\leq n\sum_{1\leq q\leq Q}\tau_3(q)^{2m-2}=n Q M(\tau_3^{2m-2};Q).\]
Noting that $M_t\le M_{2m}=M_{\mathrm{log}}(\tau_3^{2m};Q)$ for each $t\le 2m$, this completes the proof of Proposition~\ref{th:rat}. Thus we are left to establish \eqref{eq:InductionTarget}.

We first observe that, since $\ind{\B}^{(\ast j)}(c/q)=\sum_{b/\ell\in \B}\ind{\B}^{(\ast j-1)}(c/q-b/\ell)$, we have
\begin{align*}
\sideset{}{'}\sum_{c/q\in\mathbb{Q}} \tau_3(q)^{2t} \ind{\B}^{(\ast j)}(c/q)^2&=\sideset{}{'}\sum_{\substack{c/q=c'/q'+b/\ell\\ b/\ell\in\B}}\tau_3(q)^{2t}1_{\B}^{(\ast j)}(c/q)1_{\B}^{(\ast j-1)}(c'/q').
%&\leq \sum_x ((\tilde{\tau}_3^{2t}g)\ast \ind{-\B})(x)f(x).
\end{align*}
We let $f(x):=\ind{\B}^{(\ast j-1)}(x)$ and $g(x):=\ind{\B}^{(\ast j)}(x)$, so that in particular $g$ is supported on $j\B$. Since $\B\subset (0,1]$ we know that $g$ is supported on $(0,j]$. Furthermore, since $\B\subset\bbq_{\le Q}$ we have $\abs{\B}\leq Q^2$ so $\|f\|_\infty,\|g\|_\infty\leq Q^{2j}$. We now apply Lemma~\ref{lem-toref} with $\omega(q)=\tau_3(q)^{2t}$. This gives the upper bound
\begin{align*}
&\sideset{}{'}\sum_{c/q\in\mathbb{Q}} \tau_3(q)^{2t} \ind{\B}^{(\ast j)}(c/q)^2= \sideset{}{'}\sum_{\substack{c/q=c'/q'+b/\ell\\ b/\ell\in\B}}\omega(q)g(c/q)f(c'/q')\\
&\ll (m\log Q)^{3/2}(Q n M_{2t+1})^{1/2} \brac{\sideset{}{'}\sum_{c'/q'\in\bbq} \tau_3(q')^{2t+2}f\Bigl(\frac{c'}{q'}\Bigr)^2}^{1/2}\brac{\sideset{}{'}\sum_{c/q\in\bbq} \tau_3(q)^{2t}g\Bigl(\frac{c}{q}\Bigr)^2}^{1/2}.
\end{align*}
The left hand side is $\sum_{c/q\in\bbq} \tau_3(q)^{2t}g(c/q)^2$, so this rearranges to give the claimed bound \eqref{eq:InductionTarget}.
\end{proof}
%
%
%%%%%%%%%%%%%%%%%%%%%%%%%%%%%%%%%
%
%
\section{Basic density Increment}
%
%
%%%%%%%%%%%%%%%%%%%%%%%%%%%%%%%%%
%
%
In this section we establish a simple $L^2$ density increment lemma, which says that if there are many large Fourier coefficients which are close to rationals with the same denominator, then one can find a large arithmetic progression on which the set has increased density. Statements of this type are standard, and our lemma only differs cosmetically from similar statements used in \cite{PiStSz:1988} or \cite{RuzsaSanders}. It may be helpful to bear in mind that this will be applied with some $q,K\ll \alpha^{-O(1)}$ and $\nu\gg \alpha^{o(1)}$ (where the $o(1)\to 0$ as $\alpha\to 0$).

We introduce the notation
\[
\mathfrak{M}(\tfrac{a}{q};N,K):=\{\gamma\in (0,1] : \norm{\gamma-a/q}\leq K/qN\}
\]
to denote the major arcs which appear in the circle method. Note that these major arcs are disjoint for distinct $a/q\in \bbq_{=q}$ provided $K<N/2$. 
%
%
%%%%%%%%%%%%%%%%%%%%%%%%%%%%%%%%%
%
%
\begin{lemma}[Large Fourier coefficents with the same denominator give density increment]\label{lem-dens}
Let $\nu,\alpha\in(0,1]$ and let $N, K,q\ge 1$ be such that $K<N/2$ and $\nu\alpha N/(K q^2)$ is sufficiently large. Let $\A\subset [N]$ be a set with no non-zero square differences and density $\alpha=\abs{\A}/N$, and
\[
\sum_{\frac{a}{q}\in \bbq_{=q}}\int_{\mathfrak{M}(\frac{a}{q};N,K)}\abs{\widehat{\ind{\A}}(\gamma)-\alpha\widehat{\ind{[N]}}(\gamma)}^2\td\gamma \geq \nu\alpha \abs{\A}.
\]
Then there exists $N'\gg \nu\alpha N/(K q^2)$ and a set $\A'\subset [N']$ with no non-zero square differences such that the density $\alpha'=|\A'|/N'$ satisfies
\[\alpha'\geq (1+\nu/5)\alpha.\]
\end{lemma}
%
%
%%%%%%%%%%%%%%%%%%%%%%%%%%%%%%%%%
%
%
\begin{proof}
Let $\P=(q^2)\cdot [N']$ be an arithmetic progression of difference $q^2$ and length $N'$, for some $N'$ to be chosen later. If $\gamma\in\mathfrak{M}(a/q;N,K)$ for some $a/q$ then for any $1\le n'\le N'$ we have
\begin{align*}
\abs{1-e(\gamma q^2 n')}\ll \norm{\gamma q^2 n'}&\leq\abs{\gamma q^2 n'-a q n'}= q^2 n'\abs{\gamma-a/q}\leq \frac{q^2 N' K}{q N}.
\end{align*}
(We recall that $\norm{\cdot}$ is the distance to the nearest integer.) In particular we can ensure that $\abs{1-e(\gamma q^2k)}\leq 1/2$ provided we have
\begin{equation}
N'\leq \frac{c N}{q K}\label{eq:N'Bound}
\end{equation}
 for some sufficiently small absolute constant $c>0$. Thus if $\gamma\in\mathfrak{M}(a/q;N,K)$ and \eqref{eq:N'Bound} holds,
\[
\abs{\widehat{\ind{\P}}(\gamma)-N'}\leq \sum_{1\leq n'\leq N'}\abs{1-e(\gamma q^2 n')}\leq N'/2,
\]
and so $|\widehat{\ind{\P}}(\gamma)|\ge N'/2$. Let $g=\ind{\A}-\alpha\ind{[N]}$ be the balanced function of $\A$. It follows that (using the assumption of the lemma)
\begin{align}
\sum_{x\in\bbz}(\ind{\P}\ast g)(x)^2=\int_0^1 \Abs{\widehat{\ind{\P}}(\gamma)}^2\abs{\widehat{g}(\gamma)}^2\td\gamma&\ge \frac{(N')^2}{4}\sum_{\frac{a}{q}\in \bbq_{=q}}\int_{\mathfrak{M}(\frac{a}{q};N,K)}\abs{\widehat{g}(\gamma)}^2\td\gamma\nonumber \\
&\ge \frac{\nu \alpha(N')^2\abs{\A}}{4}.\label{eq:DiscLower}
\end{align}
On the other hand, recalling that $g=\ind{\A}-\alpha\ind{[N]}$, the left-hand side is equal to
\begin{equation}
\norm{\ind{\P}\ast \ind{\A}}_2^2-2\alpha \sum_{x\in\bbz}  (\ind{\P}\ast \ind{\A})(x)\cdot(\ind{\P}\ast \ind{[N]})(x) +\alpha^2\norm{\ind{\P}\ast \ind{[N]}}_2^2.
\label{eq:gExpanded}
\end{equation}
The third term of \eqref{eq:gExpanded} trivially satisfies
\begin{equation}
\alpha^2\norm{\ind{\P}\ast \ind{[N]}}_2^2\le \alpha^2 N \abs{\P}^2=\alpha \abs{\A}(N')^2.
\label{eq:ConstBound}
\end{equation}
For the second term of \eqref{eq:gExpanded}, we note that
\begin{align*}
\sum_{x\in\bbz}\abs{(\ind{\P}\ast \ind{-\P}\ast \ind{[N]})(x)-(N')^2\ind{[N]}(x)}
&\leq \sum_{y\in\bbz}(\ind{\P}\ast\ind{-\P})(y)\sum_{x\in\bbz}\abs{\ind{[N]}(x-y)-\ind{[N]}(x)}\\
&\ll \sum_{y\in\bbz}(\ind{\P}\ast\ind{-\P})(y) |y|\\
&\ll q^2(N')^3.
\end{align*}
In particular,
\begin{align}
\sum_{x\in\bbz}  (\ind{\P}\ast \ind{\A})(x)\cdot (\ind{\P}\ast \ind{[N]})(x) 
&= \sum_{y\in \A}\ind{\P}\ast \ind{-\P}\ast \ind{[N]}(y)\nonumber\\
&=\abs{\A}(N')^2+O(q^2(N')^3).\label{eq:SecondBound}
\end{align}
By substituting \eqref{eq:ConstBound} and \eqref{eq:SecondBound} into \eqref{eq:gExpanded}, we have
\[
\norm{\ind{\P}\ast  \ind{\A}}_2^2\geq 2\alpha\Bigl(\abs{\A}(N')^2+O(q^2(N')^3)\Big)-\alpha\abs{\A}(N')^2+\frac{\nu \alpha(N')^2\abs{\A}}{4}.
\]
Provided we have
\begin{equation}
N'\leq \frac{c \nu \alpha N}{q^2}\label{eq:N'Bound2}
\end{equation}
for some sufficiently small constant $c>0$, we see that the $O(q^2(N')^3)$ term contributes at most $\nu\alpha\abs{\A}(N')^2/100$ in total, and so
\[
\norm{\ind{\P}\ast \ind{-\A}}_2^2=\norm{\ind{\P}\ast  \ind{\A}}_2^2\geq  \Bigl(1+\frac{\nu}{5}\Bigr)\alpha \abs{\A}(N')^2.
\]
Since $\norm{\ind{\P}\ast \ind{-\A}}_1=N'\abs{\A}$ there exists some $x\in\bbz$ such that
\[\abs{(q^2\cdot [N'])\cap (\A+x)}=\ind{\P}\ast \ind{-\A}(x)\geq \Bigl(1+\frac{\nu}{5}\Bigr)\alpha N'.\]
Therefore, if we set
\[\A' := \frac{1}{q^2}\cdot ((q^2\cdot [N'])\cap (\A+x)),\]
then $\A'\subset [N']$, $\A'$ has density $\alpha'\ge (1+\nu/5)\alpha$ and $\A'$ has no non-zero square differences since translating a set by a constant or dilating a set by a square doesn't change the number of non-zero square differences. This therefore gives the result with $N'=\lfloor c\nu \alpha N/(K q^2)\rfloor$ for a suitably small absolute constant $c>0$ (since this choice satisfies \eqref{eq:N'Bound} and \eqref{eq:N'Bound2} and $N'\gg \nu\alpha N/(K q^2)$).
\end{proof}
%
%
%%%%%%%%%%%%%%%%%%%%%%%%%%%%%%%%%
%
%
\section{Large Fourier coefficients at rationals with small denominators}
%
%
%%%%%%%%%%%%%%%%%%%%%%%%%%%%%%%%%
%
%
In this section we show how to find many rationals with small denominator in the large spectrum of $\A$ (that is, the set of frequencies with large Fourier coefficient). This follows standard lines, combining the circle method with classical bounds for Weyl sums.
%
%
%%%%%%%%%%%%%%%%%%%%%%%%%%%%%%%%%
%
%
\begin{lemma}[Bounds for exponential sums over squares]\label{lem:Squares}
Let $1\le a \le q$ with $\gcd(a,q)=1$ and
\begin{align*}
\mathfrak{M}(\tfrac{a}{q};N,K)&:=\{\gamma\in (0,1] : \norm{\gamma-a/q}\leq K/qN\},\\
W(n) &:= \begin{cases} \frac{2m}{N^{1/2}},&\textrm{if }n=m^2\leq N,\\
0,&\textrm{otherwise.}\end{cases}
\end{align*}
Then we have the following bounds:
\begin{enumerate}
\item For all $\beta\in \mathbb{R}$ we have
\[\Abs{\widehat{W}(a/q+\beta)}\ll \frac{N^{1/2}}{q^{1/2}}+(q\log q)^{1/2}(1+\abs{\beta}N).\]
\item If $Kq\log q\ll N$ and $K^3\log q\ll qN$ then  
\[\int_{\mathfrak{M}(\tfrac{a}{q};N,K)}\Abs{\widehat{W}(\gamma)}^2\td\gamma \ll \frac{1}{q}.\]
\end{enumerate}
\end{lemma}
%
%
%%%%%%%%%%%%%%%%%%%%%%%%%%%%%%%%%
%
%
\begin{proof}
This is standard. By \cite[equation 8]{PiStSz:1988} (or partial summation) we have
\[\widehat{W}(a/q+\beta) = \frac{S(a;q)}{q} \widehat{W}(\beta)+O\brac{(q\log q)^{1/2}(1+\abs{\beta}N)},\]
where $S(a;q):=\sum_{1\leq n\leq q}e(an^2/q)$ is the complete Gauss sum. The classical estimate $S(a;q)\ll q^{1/2}$ for $\gcd(a,q)=1$ now gives $(1)$. Using this estimate again, we find
\[\int_{\mathfrak{M}(\tfrac{a}{q};N,K)}\Abs{\widehat{W}(\gamma)}^2\td\gamma \ll \frac{1}{q}\int_{0}^{K/qN}\Abs{\widehat{W}(\beta)}^2\td \beta + \frac{K\log q}{N}+(q\log q)N^2\int_0^{K/qN}\beta^2\td \beta.\]
The second and third summands contribute
\[\ll \frac{K\log q}{N}+(q \log q)N^2\frac{K^3}{q^3N^3}\ll  \frac{K\log q}{N}+ \frac{K^3\log q}{q^2N}\ll \frac{1}{q}\]
by our assumptions on $q$ and $K$. By \cite[equation 10]{PiStSz:1988} and the trivial bound, if $\beta\leq N^{-7/8}$ then 
\[\Abs{\widehat{W}(\beta)}\ll \min\Bigl(N^{1/2},\frac{\beta^{-1}}{N^{1/2}}\Bigr).\]
(Note that this bound is slightly better than what one gets with the unweighted sum, but could be improved further with more smoothing.) This gives
\[\frac{1}{q}\int_{0}^{K/qN}\Abs{\widehat{W}(\beta)}^2\td \beta\ll \frac{1}{q}+\frac{1}{qN}\int_{1/10N}^{K/qN}\beta^{-2}\td\beta\ll \frac{1}{q},\]
as required.
\end{proof}
%
%
%%%%%%%%%%%%%%%%%%%%%%%%%%%%%%%%%
%
%
\begin{lemma}\label{lem:Rational}
Let $\A\subset [N]$ be a set of density $\alpha=\abs{\A}/N\ge N^{-1/8}$ with no square differences. Then there exist quantities $B,Q,K$ with $\alpha^{O(1)}\ll B,Q,K\ll \alpha^{-O(1)}$ and a set $\B\subset \bbq_{\leq Q}$ such that 
\begin{enumerate}
\item For each $a/q\in \B$ there exists $\gamma_{a/q}\in (0,1]$ with $\norm{\gamma_{a/q}-a/q}\ll \alpha^{-O(1)}/N$ and
\[\sum_{\frac{a}{q}\in \B}\Abs{\widehat{\ind{\A}}(\gamma_{a/q})}\gg B\frac{\abs{\A}Q^{1/2}}{\log(1/\alpha)^{2}}.\]
\item For each $a/q\in \B$ we have, if $g=\ind{\A}-\alpha\ind{[N]}$,
\[
\int_{\mathfrak{M}(\tfrac{a}{q};N,K)}\Abs{\widehat{g}(\gamma)}^2\td\gamma\gg \frac{\alpha \abs{\A}}{B^2}.
\]
\end{enumerate}
\end{lemma}
%
%
%%%%%%%%%%%%%%%%%%%%%%%%%%%%%%%%%
%
%
\begin{proof}
We first note that, by the estimate of \cite{Sa:1978}, say, we may assume that $\alpha \ll 1/(\log N)^{1/4}$ since $\A$ has no non-zero square differences. (This is not essential to the method, but allows for cleaner bounds in the final statements.) Let
\[
W(n) := \begin{cases} \frac{2m}{N^{1/2}}&\textrm{if }n=m^2\leq N\textrm{ and }\\
0&\textrm{otherwise.}\end{cases}
\]

By orthogonality and the fact that $\A$ has no non-zero square differences, we have
\begin{align*}
\int_0^1\widehat{\ind{\A}}(\gamma)\overline{\widehat{\ind{\A}}(\gamma)}\widehat{W}(\gamma)\td\gamma
&=\sum_{a,b\in \A}\sum_{1\leq n\leq N^{1/2}}W(n^2)\int_0^1 e(\gamma(a-b+n^2))\td\gamma\\
&=\sum_{a,b\in \A}\sum_{1\leq n\leq N^{1/2}} W(n^2)1_{b-a=n^2}\\
&= 0.
\end{align*}
Suppose first that $\abs{\A\cap (N/2,N]}\geq \abs{\A}/2$. In this case, if we let $g=\ind{\A}-\alpha \ind{[N]}$, then
\begin{align*}
\int_0^1\widehat{g}(\gamma)\overline{\widehat{\ind{\A}}(\gamma)}\widehat{W}(\gamma)\td\gamma
&=-\alpha \sum_{a\in A}\sum_{1\leq y\leq N}\sum_{1\leq n\leq N^{1/2}}W(n^2)\int_0^1 e(\gamma(y-a+n^2))\td\gamma\\
&=-\alpha\sum_{a\in A}\sum_{1\leq n\leq N^{1/2}}W(n^2)1_{1\leq a-n^2\leq N}\\
&\leq  -\tfrac{1}{8}\alpha \abs{\A} N^{1/2},
\end{align*}
say, since certainly all $a\in \A\cap (N/2,N]$ and $n\leq (N/2)^{1/2}$ will satisfy $1\leq a-n^2\leq N$. If $\abs{\A\cap [1,N/2]}\geq \abs{\A}/2$, then arguing similarly, we have
\[\int_0^1\overline{\widehat{g}(\gamma)}\widehat{\ind{\A}}(\gamma)\widehat{W}(\gamma)\td\gamma
\leq  -\tfrac{1}{8}\alpha\abs{\A}N^{1/2}.\]
Thus in either case, we have
\begin{equation}\label{eqtoref2}
\int_0^1\Abs{\widehat{g}(\gamma)\widehat{\ind{\A}}(\gamma)\widehat{W}(\gamma)}\td\gamma
\geq \tfrac{1}{8}\alpha\abs{\A}N^{1/2}.
\end{equation}
By Dirichlet's theorem on Diophantine approximation, given any choice of $K\ge 1$, every $\gamma\in[0,1]$ satisfies $\norm{\gamma-a/q}<K/(N q)$ for some $1\leq q\leq N/K$ and $1\leq a\leq q$ with $\gcd(a,q)=1$. If this holds for some $q>K$ and $K\leq N^{1/2}$, say, then by Lemma~\ref{lem:Squares},
\[\Abs{\widehat{W}(\gamma)}\ll \brac{\frac{N(\log N)}{K}}^{1/2}.\]
If we choose
\begin{equation}
K:=\lceil C\alpha^{-2}\log N\rceil
\label{eq:KDef}
\end{equation}
 for some suitably large absolute constant $C>0$, then we see that $\Abs{\widehat{W}(\gamma)}\leq \alpha N^{1/2}/32$ for such $\gamma$. The contribution to \eqref{eqtoref2} from these $\gamma$ is thus at most 
\begin{align}
\frac{\alpha N^{1/2}}{32}\int_0^1\Abs{\widehat{g}(\gamma)\widehat{\ind{\A}}(\gamma)}\td\gamma
&\leq \frac{\alpha N^{1/2}}{32}\Bigl(\int_0^1\Abs{\widehat{\ind{\A}}(\gamma)}^2\td\gamma+\int_0^1\abs{\alpha\widehat{\ind{[N]}}(\gamma)\widehat{\ind{\A}}(\gamma)}\td\gamma\Bigr)\nonumber\\
&\leq \frac{ \alpha\abs{\A} N^{1/2}}{16}.\label{eq:MinorArc}
\end{align}
(Here we used the triangle inequality in the first line, and the Cauchy-Schwarz inequality and Parseval's identity in the second.) We recall that for $\gcd(a,q)=1$
\[
\mathfrak{M}(a/q)=\mathfrak{M}(a/q;N,K) := \{ \gamma\in (0,1] : \norm{\gamma-a/q}\leq K/q N\},
\]
and note that with our choice $K=\lceil C\alpha^{-2}\log N\rceil$ these sets are distinct for $q\le K$ and $\gcd(a,q)=1$ since $\alpha\ge N^{-1/3}$ and $N$ is sufficiently large. Therefore, combining \eqref{eqtoref2} and \eqref{eq:MinorArc}, we find
\begin{equation}
\sum_{\tfrac{a}{q}\in \bbq_{\leq K}}\int_{\mathfrak{M}(\tfrac{a}{q})}\Abs{\widehat{g}(\gamma)\widehat{\ind{\A}}(\gamma)\widehat{W}(\gamma)}\td\gamma\geq \frac{\alpha \abs{\A}N^{1/2}}{16}.
\end{equation}
By the Cauchy-Schwarz inequality and Lemma \ref{lem:Squares}, we have
\begin{align*}
\int_{\mathfrak{M}(\tfrac{a}{q})}&\Abs{\widehat{g}(\gamma)\widehat{\ind{\A}}(\gamma)\widehat{W}(\gamma)}\td\gamma\\
&\ll \Bigl(\int_{\mathfrak{M}(\frac{a}{q})} \abs{\widehat{g}(\gamma)}^2\td \gamma\Bigr)^{1/2}\Bigl(\int_{\mathfrak{M}(\frac{a}{q})} \abs{\widehat{W}(\gamma)}^2\td \gamma\Bigr)^{1/2}\sup_{\gamma\in\mathfrak{M}(\tfrac{a}{q})}\abs{\widehat{\ind{\A}}(\gamma)}\\
&\ll\frac{1}{q^{1/2}}\Bigl(\int_{\mathfrak{M}(\frac{a}{q})} \abs{\widehat{g}(\gamma)}^2\td \gamma\Bigr)^{1/2}\sup_{\gamma\in\mathfrak{M}(\tfrac{a}{q})}\abs{\widehat{\ind{\A}}(\gamma)}.
\end{align*}
Therefore 
\begin{equation}\label{eq-toref}
\sum_{\tfrac{a}{q}\in \bbq_{\leq K}}\frac{1}{q^{1/2}}\Bigl(\int_{\mathfrak{M}(\tfrac{a}{q})}\Abs{\widehat{g}(\gamma)}^2\td\gamma\Bigr)^{1/2}\Bigl(\sup_{\gamma\in \mathfrak{M}(\tfrac{a}{q})}\abs{\widehat{\ind{\A}}(\gamma)}\Bigr)\gg \alpha \abs{\A} N^{1/2}.
\end{equation}
Let $\Gamma_1$ be the set of $a/q\in\bbq_{\le K}$ for which
\begin{equation}
\int_{\mathfrak{M}(\tfrac{a}{q})}\Abs{\widehat{g}(\gamma)}^2\td\gamma\le \frac{N}{K^5}.
\label{eq:SmallInt}
\end{equation}
Since $|\Gamma_1|\le |\bbq_{\le K}|\le K^2$ and $|\widehat{\ind{\A}}(\gamma)|\le |\A|$, the contribution to \eqref{eq-toref} from $\gamma\in\Gamma_1$ is
\[
\ll \sum_{\tfrac{a}{q}\in \bbq_{\leq K}} \frac{1}{q^{1/2}}\cdot \frac{N^{1/2}}{K^{5/2}} \cdot |\A|  \ll \frac{\alpha \abs{\A}N^{1/2} }{(\log{N})^{1/2}}.
\]
Thus we may restrict our attention to the set $\Gamma_2=\bbq_{\le K}\backslash\Gamma_1$ of $a/q\in\bbq_{\le K}$ for which \eqref{eq:SmallInt} does not hold. Indeed, we have
\begin{equation}
\sum_{\tfrac{a}{q}\in\Gamma_2}\frac{1}{q^{1/2}}\Bigl(\int_{\mathfrak{M}(\tfrac{a}{q})}\Abs{\widehat{g}(\gamma)}^2\td\gamma\Bigr)^{1/2}\Bigl(\sup_{\gamma\in \mathfrak{M}(\tfrac{a}{q})}\abs{\widehat{\ind{\A}}(\gamma)}\Bigr)\gg \alpha \abs{\A} N^{1/2}.
\end{equation}
Since $\Abs{\widehat{g}(\gamma)}\le 2\abs{\A}\le 2N$ and $\mathrm{meas}(\mathfrak{M}(a/q))\ll K/N$, we see that for any $\gamma\in\bbq_{\le K}$
\[
\Bigl(\int_{\mathfrak{M}(\tfrac{a}{q})}\Abs{\widehat{g}(\gamma)}^2\td\gamma\Bigr)^{1/2}\ll K^{1/2}N^{1/2},
\]
and so, comparing to \eqref{eq:SmallInt}, for $\gamma\in\Gamma_2$ the quantity on the left-hand side is $N^{1/2}K^{O(1)}$. Therefore, by dyadic pigeonholing, there are some quantities $B,Q$ with $\alpha^{-1}K^{-1/2}\ll B\ll \alpha^{-1}K^{5/2}$ and $1\le Q\le K$, together with a set $\B\subset\Gamma_2$, such that
\begin{enumerate}
\item If $a/q\in\B$ with $\gcd(a,q)=1$ then $q\in[Q,2Q]$.
\item For all $\gamma\in\B$ we have 
\[
\frac{\alpha N^{1/2}}{B}\leq \Bigl(\int_{\mathfrak{M}(\tfrac{a}{q})}\Abs{\widehat{g}(\gamma)}^2\td\gamma\Bigr)^{1/2}\leq \frac{2 \alpha N^{1/2}}{B}.
\]
\item We have
\[
\sum_{\tfrac{a}{q}\in\B}\frac{1}{q^{1/2}}\Bigl(\int_{\mathfrak{M}(\tfrac{a}{q})}\Abs{\widehat{g}(\gamma)}^2\td\gamma\Bigr)^{1/2}\Bigl(\sup_{\gamma\in \mathfrak{M}(\tfrac{a}{q})}\abs{\widehat{\ind{\A}}(\gamma)}\Bigr)\gg \frac{\alpha \abs{\A} N^{1/2}}{(\log{K})^2}.
\]
\end{enumerate}
Recalling that $K\ll\alpha^{-O(1)}$ (since we are assuming that $\alpha\le 1/(\log{N})^{1/4}$), and letting $\gamma_{a/q}$ be the point in $\mathfrak{M}(a/q)$ where $|\widehat{\ind{\A}}(\gamma)|$ attains its maximum, we see that this gives the result.
\end{proof}
%
%
%%%%%%%%%%%%%%%%%%%%%%%%%%%%%%%%%
%
%
Combining Lemma \ref{lem:Rational} with Lemma \ref{lem-dens} gives the following
%
%
%%%%%%%%%%%%%%%%%%%%%%%%%%%%%%%%%
%
%
\begin{lemma}\label{lem:TripleAlternative}
Let $\nu>0$. Let $\A\subset [N]$ be a set of density $\alpha=\abs{\A}/N$ with no square differences. Then at least one of the following holds:
\begin{enumerate}
\item ($\A$ is sparse) $\log(1/\alpha)\gg \log N$.
\item (There is a density increment) There is some $N'\gg \nu \alpha^{O(1)}N$ and $A'\subset [N']$ with no non-zero square differences, which has density
\[\alpha'\geq (1+\nu/5)\alpha.\]
\item (There are many large Fourier coefficients close to rationals of different denominators) There are $B,Q\ll \alpha^{-O(1)}$ and a set $\B\subset \bbq_{\leq Q}$ such that both of the following hold:
\begin{enumerate}
\item For each $a/q\in \B$ there exists $\gamma_{a/q}\in (0,1]$ such that $\norm{\gamma_{a/q}-a/q}\ll \alpha^{-O(1)}/N$ and
\[\sum_{\frac{a}{q}\in \B}\Abs{\widehat{\ind{\A}}(\gamma_{a/q})}\gg \frac{B\abs{\A}Q^{1/2}}{\log(1/\alpha)^{O(1)}}.\]
\item For every $1\leq q\leq Q$ we have
\[\abs{ \B\cap \bbq_{=q}}\leq \nu B^2.\]
\end{enumerate}
\end{enumerate}
\end{lemma}
%
%
%%%%%%%%%%%%%%%%%%%%%%%%%%%%%%%%%
%
%
\begin{proof}
Assume that neither $(1)$ nor $(2)$ hold, so we wish to establish $(3)$. Since $(2)$ does not hold, by  Lemma \ref{lem-dens}, we have that for any $q$
\[
\sum_{\frac{a}{q}\in \bbq_{=q}}\int_{\mathfrak{M}(\frac{a}{q}:N,K)}|\widehat{g}(\gamma)|^2\td\gamma \leq \nu\alpha \abs{\A}.
\]
for any $K\ll \alpha^{-O(1)}$.
Since $(1)$ does not hold, we have $\alpha\ge N^{-1/8}$ and so $\A$ satisfies the conditions of Lemma \ref{lem:Rational}. Thus, by Lemma \ref{lem:Rational} there are $B,Q,K\ll\alpha^{-O(1)}$ and $\B\subset\bbq_{\le Q}$ such that for all $a/q\in\B$ there exists $\gamma_{a/q}$ such that $\| \gamma_{a/q}-a/q\|\ll \alpha^{-O(1)}/N$ and 
\[
\sum_{\frac{a}{q}\in \B}\Abs{\widehat{\ind{\A}}(\gamma_{a/q})}\gg B\frac{\abs{\A}Q^{1/2}}{\log(1/\alpha)^{O(1)}},
\]
and for all $a/q\in\B$
\[
\int_{\mathfrak{M}(\tfrac{a}{q};N,K)}\Abs{\widehat{g}(\gamma)}^2\td\gamma\gg \frac{\alpha \abs{\A}}{B^2}.
\]
Summing this second inequality over $a/q\in\B\cap\bbq_{=q}$ we see that
\[
\frac{\alpha \abs{\A}}{B^2}|\B\cap\bbq_{=q}|\ll \sum_{\tfrac{a}{q}\in \B\cap\bbq_{=q}} \int_{\mathfrak{M}(\tfrac{a}{q};N,K)}\Abs{\widehat{g}(\gamma)}^2\td\gamma\ll \nu\alpha \abs{\A}.
\]
Thus $|\B\cap\bbq_{=q}|\ll \nu B^2$, as required.
\end{proof}
%
%
%%%%%%%%%%%%%%%%%%%%%%%%%%%%%%%%%
%
%
\section{Refined density increment and proof of Theorem \ref{thmain}}
%
%
%%%%%%%%%%%%%%%%%%%%%%%%%%%%%%%%%
%
%
We will now show that there cannot be a large set of rationals with distinct denominators each of which has a large Fourier coefficient. This relies on Theorem \ref{th:main} which shows that there is a lack of additive structure amongst such rationals, and a variant of Chang's lemma \cite{Chang} (or its predecessors such as the Montgomery-Hal\'{a}sz method \cite{Montgomery}) which shows that any large set of frequencies with large Fourier coefficients must have some additive structure, and is the key way in which our argument differs from previous approaches. Ultimately this will show that for a suitable choice of parameter $\nu$, the third possibility in Lemma~\ref{lem:TripleAlternative} cannot occur, and Lemma \ref{lem:TripleAlternative} can be refined to give a density increment. An iterative application of this density increment then proves our main result, Theorem \ref{thmain}.

\begin{lemma}[Variant of Chang's Lemma]\label{lem:Chang}
Let $\eta>0$, let $\A\subset[N]$ be a set of density $\alpha=\abs{\A}/N$ and let $\B\subset(0,1]$. Then, for each $m\ge 1$,
\[
\sum_{b\in\B}|\widehat{\ind{\A}}(b)|\ll |\A|\alpha^{-1/2m}E_{2m}(\B;1/2N)^{1/2m},
\]
where the approximate additive energy $E_{2m}(\C;\delta)$ is given by
\[
E_{2m}(\C;\delta):=|\{b_1,\dots,b_{2m}\in\C:\,\|b_1+\dots+b_m-b_{m+1}\dots-b_{2m}\|\le \delta\}|
\]
(where we recall that $\|\cdot \|$ denotes the distance to the nearest integer).
\end{lemma}
%
%
%%%%%%%%%%%%%%%%%%%%%%%%%%%%%%%%%
%
%
\begin{proof}
Let $\theta_b\in\bbr$ be a phase such that $e(\theta_b)\widehat{\ind{\A}}(b)=|\widehat{\ind{\A}}(b)|\in\bbr_{>0}$. Then, by H\"{o}lder's inequality, we have
\begin{align}
\sum_{b\in\B}|\widehat{\ind{\A}}(b)|&= \sum_{b\in\B}e(\theta_b)\sum_{a\in\A}e(ab)\nonumber\\
&\le \Bigl(\sum_{a\in\A}1\Bigr)^{1-1/2m}\Bigl(\sum_{a\in\A}\Bigl|\sum_{b\in B}e(\theta_b+b a)\Bigr|^{2m}\Bigr)^{1/2m}.\label{eq:HolderBound}
\end{align}
Let $\psi(t):=\sin(\pi t)^2/(\pi t)^2$ so that $\widehat{\psi}(\xi)=\int_{-\infty}^{\infty}\psi(t)e^{-2\pi i \xi t}\td t$ satisfies $\widehat{\psi}(\xi)=1-|\xi|$ for $|\xi|\le 1$ and $\widehat{\psi}(\xi)=0$ for $|\xi|>1$. Since $\psi(t)\ge 0$ and $\psi(t)\ge 4/\pi^2\ge 1/3$ on $[0,1/2]$ we see that
\begin{align*}
\sum_{a\in\A}&\Bigl|\sum_{b\in B}e(\theta_b+ba)\Bigr|^{2m}\le 3\sum_{n\in\bbz}\psi\Big(\frac{n}{2N}\Bigr)\Bigl|\sum_{b\in B}e(\theta_b+ba)\Bigr|^{2m}\\
&\le 3\sum_{b_1,\dots,b_{2m}\in\B}\Bigl|\sum_{n\in\bbz}\psi\Bigl(\frac{n}{2N}\Bigr)e\Bigl(n(b_1+\dots+b_m-b_{m+1}\dots -b_{2m})\Bigr)\Bigr|
\end{align*}
Applying Poisson summation to the inner sum, and recalling that $\hat{\psi}$ is supported on $[-1,1]$, we see that this is equal to
\begin{align*}
&6\sum_{b_1,\dots,b_{2m}\in\B}N\Bigl|\sum_{h\in \bbz}\hat{\psi}\Bigl(2N(b_1+\dots+b_m-b_{m+1}\dots-b_{2m}-h)\Bigr)\Bigr|\\
&\ll N|\{b_1,\dots,b_{2m}\in\B:\,\|b_1+\dots+b_m-b_{m+1}\dots-b_{2m}\|\le 1/2N\}|.
\end{align*}
Substituting this into \eqref{eq:HolderBound} and rearranging then gives the result.
\end{proof}
%
%
%%%%%%%%%%%%%%%%%%%%%%%%%%%%%%%%%
%
%
\begin{lemma}\label{lem-toit}
Let $A\subset [N]$ be a set of density $\alpha=\abs{A}/N$ with no non-zero square differences. There exists an absolute constant $c>0$ such that if
\[\nu = \exp\brac{-c \frac{\log(1/\alpha)}{\log\log(1/\alpha)}}\]
then either
\begin{enumerate}
\item $\log(1/\alpha) \gg \log N/\log\log N$, or
\item there are $N'\gg \alpha^{O(1)}N$ and $A'\subset [N']$ with no non-zero square differences, which has density
\[\alpha'\geq (1+\nu/5)\alpha.\]
\end{enumerate}
\end{lemma}
%
%
%%%%%%%%%%%%%%%%%%%%%%%%%%%%%%%%%
%
%
\begin{proof}
As before, we may assume that $\log N \ll \alpha^{-O(1)}$, by the result of \cite{Sa:1978}. We assume that $(1)$ and $(2)$ do not hold, and hope to arrive at a contradiction, for a suitable choice of $\nu$.

We apply Lemma~\ref{lem:TripleAlternative}, and find that we must be in the third case since otherwise $(1)$ or $(2)$ would hold. Thus there are $B,Q\ll \alpha^{-O(1)}$ and a set $\B\subset \bbq_{\leq Q}$ such that
for each $a/q\in \B$ there exists $\gamma_{a/q}=a/q+O(\alpha^{-O(1)}/N)$ satisfying
\[\sum_{\frac{a}{q}\in \B}\Abs{\widehat{\ind{\A}}(\gamma_{a/q})}\gg \frac{B\abs{\A}Q^{1/2}}{\log(1/\alpha)^{O(1)}},\]
and for every $1\leq q\leq Q$ we have $\abs{ \B\cap \bbq_{=q}}\leq \nu B^2$. By the pigeonhole principle, there must exist some $\B'\subset \B$ which is contained in an interval of width at most $1/8m$ such that
\[\sum_{\frac{a}{q}\in \B'}\Abs{\widehat{\ind{\A}}(\gamma_{a/q})}\gg \frac{B\abs{\A}Q^{1/2}}{m\log(1/\alpha)^{O(1)}}.\]

Let $m\geq 2$ be some integer to be chosen later. We now apply Lemma \ref{lem:Chang}, which shows that for $\Gamma:=\{\gamma_b:\,b\in\B'\}$
\[
\frac{B \abs{\A}Q^{1/2}}{m\log(1/\alpha)^{O(1)}}\ll \abs{\A}\alpha^{-1/2m}E_{2m}(\Gamma;1/2N)^{1/2m}.
\]
Since $\B'$ is contained in an interval of width at most $1/8m$ we know that $b_1+\cdots-b_{2m}\in [-1/4,1/4]$. In particular, since $\abs{\gamma_b-b}\ll \alpha^{-O(1)}/N$ for $b\in \B$, provided $m\alpha^{-O(1)}< cN$ for some sufficiently small $c>0$, we have $\gamma_{b_1}+\cdots-\gamma_{b_{2m}}\in (-1/2,1/2)$, and so
\[E_{2m}(\Gamma;1/2N) = \lvert \{ b_1,\ldots,b_{2m}\in \B' : \lvert \gamma_{b_1}+\cdots-\gamma_{b_{2m}}\rvert \leq 1/2N\}\rvert.\] 

Furthermore, since $\B\subset\bbq_{\le Q}$, $b_1+\cdots-b_{2m}$ is always a rational of denominator at most $Q^{2m}$ for $b_1,\dots,b_{2m}\in\B$. Therefore if $b_1+\cdots-b_{2m}$ is not zero then it is at least $Q^{-2m}$ in absolute value. As before, since $\abs{\gamma_b-b}\ll \alpha^{-O(1)}/N$, provided $m Q^{O(m)}\alpha^{-O(1)}< cN$ for some small constant $c>0$, it follows that for any $\gamma_{b_1},\dots,\gamma_{b_{2m}}\in\Gamma$ either $\abs{\gamma_{b_1}+\cdots-\gamma_{b_{2m}}}\ge Q^{-2m}/2$ or $b_1+\cdots-b_{2m}=0$. Therefore, provided $m Q^{O(m)}\alpha^{-O(1)}<cN$ for sufficiently small $c>0$, the approximate additive energy $E_{2m}(\Gamma;1/N)$ actually only counts the times when the corresponding sum of rationals is zero, so
\[
E_{2m}(\Gamma;1/2N)=E_{2m}(\B').
\]
Recalling that $Q\ll \alpha^{-O(1)}$, we have shown that either $\alpha^{-O(m)}\gg N$ or
\begin{equation}
E_{2m}(\B')  \gg \alpha \brac{\frac{B Q^{1/2}}{m\log(1/\alpha)^{O(1)}}}^{2m}.\label{eq:EnergyBound}
\end{equation}
We impose the condition $m\ll \log\log(1/\alpha)$, so that $\alpha^{-O(m)}=o(N)$ since we are assuming that $(1)$ does not hold, so we have \eqref{eq:EnergyBound}. We now apply Theorem~\ref{th:main} to bound $E_{2m}(\B')$ from above, which shows that for some absolute constant $C>0$ we have
\begin{equation}
m^{O(m)}(\log Q)^{C^m}(\nu B^2 Q)^{m}\gg \alpha \brac{\frac{ B Q^{1/2}}{\log(1/\alpha)^{O(1)}}}^{2m}.\label{eq:ThmBound}
\end{equation}
Here we used the assumption that $|\B'\cap\mathbb{Q}_{=q}|\le \lvert \B\cap \mathbb{Q}_{=q}\rvert \leq \nu B^2$ for all $q$ since $(2)$ does not hold.

The key feature of this bound is that the powers of $B$ and $Q$ exactly cancel, and in particular the lower bound on $\nu$ in terms of $\alpha$ is only of order $\alpha^{-O(1/m)}\log(1/\alpha)^{O(1)}$. We will derive a contradiction from this bound with a suitable choice of $\nu$, thereby proving the lemma. First note that we can rewrite \eqref{eq:ThmBound} as 
\[
\nu \gg \frac{\alpha^{1/m}}{m^{O(1)}\log(1/\alpha)^{O(1)}}\exp\brac{-C^m\frac{\log \log Q}{m}}.\]
Since $Q\ll \alpha^{-O(1)}$, if we choose $m=\lceil c'\log\log(1/\alpha)\rceil$, for some sufficiently small constant $c'>0$, then this gives 
\[\nu \gg \exp\brac{-O\brac{\frac{\log(1/\alpha)}{\log\log(1/\alpha)}}},\] 
which gives a contradiction for a suitable choice of the constant $c$ in our definition of $\nu$. This completes the proof.
\end{proof}
%
%
%%%%%%%%%%%%%%%%%%%%%%%%%%%%%%%%%
%
%
We may now finish the proof of our main theorem with an iterative application of Lemma~\ref{lem-toit}. 
%
%
%%%%%%%%%%%%%%%%%%%%%%%%%%%%%%%%%
%
%
\begin{proof}[Proof of Theorem~\ref{thmain}]
Suppose that $\A\subset [N]$ has density $\alpha=\abs{\A}/N$ and has no square differences. We wish to show that
\[\log(1/\alpha) \gg (\log\log N)(\log\log\log N).\]
Let 
\[\nu := \exp\brac{-c \frac{\log(1/\alpha)}{\log\log(1/\alpha)}}\]
be as in Lemma \ref{lem-toit}. If $\log(1/\alpha) \gg \log N/\log\log N$ then we are done. Otherwise, by Lemma~\ref{lem-toit}, there are $N'\geq \alpha^{O(1)}N$ and $\A'\subset [N']$ which has no square differences, with density 
\[\alpha' \geq (1+\nu/5)\alpha.\]
Repeatedly applying Lemma \ref{lem-toit}, we obtain some sequence $N_1,\ldots,N_t$ of integers and associated sets $\A_t\subset [N_t]$ such that
\begin{enumerate}
\item Each set $\A_t$ has no non-zero square differences.
\item $\A_t\subset [N_t]$ has density $\alpha_t=|\A_t|/N_t \geq (1+\nu/5)^t\alpha$.
\item We have $N_t \geq \alpha^{O(t)}N$.
\end{enumerate}
This process can only terminate if $N_t<N^{1/2}$, since otherwise all conditions of Lemma \ref{lem-toit} remain satisfied. However, the density of any set can never exceed $1$, so we must have $\alpha(1+\nu/20)^t\le \alpha_t\le 1$, which implies that
\[
t\ll \nu^{-1}\log(1/\alpha).
\]
Therefore
\[
N^{1/2}>N_t\ge \alpha^{O(t)}N\gg N\exp\Bigl(-O\Bigl(\nu^{-1}\log(1/\alpha)^2\Bigr)\Bigr).
\]
Thus
\[\log N \ll \nu^{-1}(\log 1/\alpha)^2.\]
Recalling that $\log(1/\nu)\ll \log(1/\alpha)/\log\log(1/\alpha)$, taking logarithms of both sides and rearranging yields
\[
\frac{\log(1/\alpha)}{\log\log(1/\alpha)}\gg \log\log{N}.
\]
This implies $\log(1/\alpha)\gg (\log\log{N})(\log\log\log{N})$, which gives the result.
\end{proof}
%
%
%%%%%%%%%%%%%%%%%%%%%%%%%%%%%%%%%
%
%


\begin{thebibliography}{10}

\bibitem{BaPePiSz:1994}
A. Balog, J. Pelikan, J. Pintz, and E. Szemer\'{e}di,
\textit{Difference sets without $k$th powers},
Acta. Math. Hungar. 65 (2) (1994), 165--187.

\bibitem{BDFKK}
J. Bourgain, S. J. Dilworth, K. Ford, S. V. Konyagin, and D. Kutzarova.
\textit{Breaking the $k^2$ barrier for explicit RIP matrices},
STOC'11 -- Proceedings of the 43rd ACM Symposium on Theory of Computing (2011), 637--644.

\bibitem{BloomSisask}
T. F. Bloom and O. Sisask,
\textit{Breaking the logarithmic barrier in Roth's theorem on arithmetic progressions},
Arxiv preprint available at \url{https://arxiv.org/abs/2007.03528}.

\bibitem{Chang}
M.-C. Chang,
\textit{A polynomial bound in Freiman's theorem},
Duke Math. J. 113 (3) (2002), 399--419.

\bibitem{Fu:1977}
H. Furstenberg,
\textit{Ergodic behavior of diagonal measures and a theorem of {S}zemer\'{e}di on arithmetic progressions},
J. Analyse Math. 31 (1977), 204--256. 

\bibitem{Gr:2017}
B. Green,
\textit{S\'{a}rkz\"{o}zy's theorem in function fields},
Q. J. Math. 68 (2017), no. 1, 237--242. 

\bibitem{Le:2015}
M. Lewko,
\textit{An improved lower bound related to the {F}urstenberg-{S}\'{a}rk\"{o}zy theorem},
Electron. J. Combin. 22 (2015), Paper 1.32, 6.

\bibitem{Maynard}
J. Maynard,
\textit{Fractional parts of polynomials},
Arxiv preprint available at \url{https://arxiv.org/abs/2011.12275}.

\bibitem{Montgomery}
H. L. Montgomery, 
\textit{Mean and large values of {D}irichlet polynomials},
Invent. Math. 8 (1969), 334--345.

\bibitem{PiStSz:1988}
J. Pintz, W. L. Steiger, and E. Szemer\'{e}di, 
\textit{On sets of natural numbers whose difference set contains no squares}, 
J. London Math. Soc. (2) 37 (1988), 219--231.

\bibitem{Ri}
A. Rice,
\textit{A maximal extension of the best-known bounds for the {F}urstenberg-{S}\'{a}rk\"{o}zy theorem},
Acta Arith. 187 (2019), 1--41.

\bibitem{Ru:1984}
I. Ruzsa, 
\textit{Difference sets without squares},
Period. Math. Hungar. 15 (1984), 205--209. 

\bibitem{RuzsaSanders}
 I. Ruzsa and T. Sanders, 
\textit{Difference sets and the primes},
Acta Arith.  131 (3) (2008) 281--301.

\bibitem{Sa:1978}
A. S\'{a}rk\"{o}zy,
\textit{On difference sets of sequences of integers. {I}},
Acta Math. Acad. Sci. Hungar. 31 (1978), 125--149. 

\bibitem{Wo}
J. Wolf,
\textit{Arithmetic structure in sets of integers},
Doctoral thesis, University of Cambridge, 2008. \url{https://doi.org/10.17863/CAM.16214}.

\end{thebibliography}
\end{document}